\documentclass[10pt,a4paper]{amsart}

\usepackage{amssymb}
\usepackage{amsthm}
\usepackage{amsmath}
\usepackage{comment}
\usepackage[dvipdfmx]{graphicx}
\newtheorem{thm}{Theorem}[section]
\newtheorem{cor}[thm]{Corollary}

\newtheorem{prop}[thm]{Proposition}

\newtheorem{knowncor}{Corollary}[section]

\newtheorem{knownthm}[knowncor]{Theorem}
\newtheorem{knownlem}[knowncor]{Lemma}

\theoremstyle{definition}
\newtheorem{definition}[knowncor]{Definition}
\newtheorem{rem}[thm]{Remark}

\newcommand{\CC}{\widehat{\mathbb{C}}}%riemann surface
\newcommand{\C}{\mathbb{C}}%complex plane
\newcommand{\D}{\mathbb{D}}%unit disk
\newcommand{\R}{\mathbb{R}}%real number
\newcommand{\A}{\mathcal{A}}%analytic functions
\renewcommand{\S}{\mathcal{S}}%univalent functions
\newcommand{\CTC}{\mathcal{C}}%clsoe-to-convex functions
\renewcommand{\R}{\mathcal{R}}%Noshiro-Warschawski functions
\newcommand{\K}{\mathcal{K}}%convex functions
\newcommand{\LU}{\mathcal{LU}}%locally univalent functions
\newcommand{\ZF}{\mathcal{ZF}}%nonzero holomorphic functions

\newcommand{\no}{\noindent}
\newcommand{\dstyle}{\displaystyle}
\renewcommand{\Re}{\textup{Re}\,}
\renewcommand{\Im}{\textup{Im}\,}
\newcommand{\closure}{\overline}
\newcommand{\vareps}{\varepsilon}
\newcommand{\round}{\partial}

\title[quasiconformal extendibility of integral transforms]{Quasiconformal extendibility of integral transforms of Noshiro-Warschawski functions}
\author[I. Hotta]{Ikkei Hotta}
\author[L.-M. Wang]{Li-Mei Wang}
\thanks{The first author was supported by Grant-in-Aid for JSPS Fellows (13J02250) and Grant-in-Aid for Young Scientists (B) (26800053)}
\thanks{The second author was supported by National Natural Science Foundation of China (No. 61374088)}
\address{Department of Applied Science, Faculty of Engineering, Yamaguchi University, 2-16-1 Tokiwadai, Ube 755-8611, JAPAN}
\email{ikkeihotta@gmail.com}
\address{School of Statistics, University of International Business and Economics, No. 10, Huixin Dongjie, Chaoyang District, Beijing, 100029 China}
\email{wangmabel@163.com}

\begin{document}

	%	++++++++++++++++++++++++++++++++++++++++++++++++++++++
	%
	%		Abstract
	%
	%	++++++++++++++++++++++++++++++++++++++++++++++++++++++

\begin{abstract}
 
Since the nonlinear integral transforms $J_{\alpha}[f](z) = \int_{0}^{z}(f'(u))^{\alpha} du$ and $I_{\alpha}[f](z) =\int_0^z (f(u)/u)^{\alpha} du$ with a complex number $\alpha$ have been introduced, a great number of studies were dedicated to deriving sufficient conditions for univalence on the unit disk.
%On the other hand, little is studied about their quasiconformal extendibility to the complex plane. %, namely, the condition that $J_{\alpha}[f](\D)$ and $I_{\alpha}[f](\D)$ are quasidisks.
On the other hand, little is known about the conditions that $J_{\alpha}[f]$ or $I_{\alpha}[f]$ produces a holomorphic univalent function in the unit disk which extends to a quasiconformal map on the complex plane.
In this paper we discuss quasiconformal extendibility of the integral transforms $J_{\alpha}[f]$ and $I_{\alpha}[f]$ for holomorphic functions which satisfy the Noshiro-Warschawski criterion.
Various approaches using pre-Schwarzian derivatives, differential subordinations and Loewner theory are taken to this problem.
 
\end{abstract}

\maketitle

	%	++++++++++++++++++++++++++++++++++++++++++++++++++++++
	%
	%		Section 1
			\section{Introduction}

	%	++++++++++++++++++++++++++++++++++++++++++++++++++++++

	%	+++++++++++++++++++++++++++++++++++++++++++
	%
	%		1-1
			\subsection{Integral Transforms}
			\label{1.1}
	%
	%	+++++++++++++++++++++++++++++++++++++++++++

Let $\A$ be the family of analytic functions defined in $\D := \{z \in \C : |z| < 1\}$ with $f(0) = 0$ and $f'(0) = 1$.
Let $\LU$ and $\ZF$ be the subclasses of $\A$ defined by $\LU := \{ f \in \A : f'(z) \neq 0,\,\,\forall z \in \D\}$ and $\ZF := \{f \in \A : f(z)/z \neq 0,\,\, \forall z \in \D\}$.

In 1915, Alexander \cite{Alexander:1915} first observed the integral transform defined by 
$
J[f](z) = \int_{0}^{z}f(u)/u\,du
$
on the class $\ZF$ maps the class of starlike functions onto the class of convex functions.
Thus one might expect that $J[f]$ always produces a univalent function for all $f \in \S$, where $\S$ is the subclass of $\A$ consisting of univalent functions on $\D$. 
However in 1963, Krzy{\.z} and Lewandowski \cite{KrzyzLewandowski:1963} gave the counterexample $f(z) = z/(1-iz)^{1-i}$ which is $\pi/4$-spirallike but transformed to a non-univalent function.
In 1972, Kim and Merkes \cite{KimMerkes:1972} extended this type of transform by introducing a complex parameter $\alpha \in \C$ as
\begin{equation*}
J_{\alpha}[f](z) := \int_{0}^{z}\left(\frac{f(u)}{u}\right)^{\alpha} du
\end{equation*}
for $f \in \ZF$, where the branch is chosen so that $(f(z)/z)^{\alpha} =1$ for $z=0$.
In their investigation it was shown $J_{\alpha}[\S] \subset \S$ when $|\alpha| \leq 1/4$ while $J_{\alpha}[\S] \not\subset \S$ if $|\alpha| > 1/2$ and $\alpha \neq 1$ (consider $J_{\alpha}[K](z)$ and Royster's example \cite{Royster:1965}, where $K(z) := z/(1-z)^{2}$ is the Koebe function).

Another object of investigation in the studies of integral transforms is $I_{\alpha}[f]$, defined by
\begin{equation}\label{Ia[f]}
I_{\alpha}[f](z) := \int_0^z (f'(u))^{\alpha}du.
\end{equation}
on $\LU$, where the branch of $(f')^{\alpha} = \exp (\alpha \log f')$ is chosen  so that $(f')^{\alpha}(0) = 1$.
Then $J_{\alpha}[f]$ is represented by $J_{\alpha}[f] = I_{\alpha}[J[f]]$.
In 1975, Pfaltzgraff \cite{Pfaltzgraff:1975} proved that $I_{\alpha}[\S] \subset \S$ if $|\alpha| \leq 1/4$.
On the other hand, Royster's example again shows that there exists a function $f \in \S$ such that $I_{\alpha}[f] \not\in \S$ if $|\alpha| > 1/3$ or $\alpha \neq 1$.

Up to now, no better estimates of the range of $|\alpha|$ have been obtained in the problems of univalence of $I_{\alpha}[f]$ and $J_{\alpha}[f]$.
The reader may be referred to \cite{Duren:1983} for basic terminology in the theory of univalent functions and \cite[Chapter 15]{Goodman:1983b} for the basic information about integral transforms on $\S$.

	%	+++++++++++++++++++++++++++++++++++++++++++
	%
	%		1-2
			\subsection{The Noshiro-Warschawski criterion}
	%
	%	+++++++++++++++++++++++++++++++++++++++++++

It is known that for a function $f \in \A$, the condition that $f'(\D)$ lies in the right half-plane ensures univalence of $f$ on $\D$.
This is referred to as the \textit{Noshiro-Warschawski criterion} due to Noshiro \cite{Noshiro:1934} and Warschawski \cite{Warschawski:1935} independently.
The original form of the theorem is the following (see also {\cite[Theorem 8]{Avkhadiev:1975}}).
\begin{knownthm}[The Noshiro-Warschawski criterion]
\label{NWthm}
A non constant function $f$ that is analytic in a convex domain $D$ is univalent in $D$ if
\begin{equation}
\label{NWcondition}
\Re\{e^{-ic} f'(z)\} \geq 0
\end{equation}
for all $z \in D$, where $c$ is a fixed real number.
\end{knownthm}

\no
As special cases, Alexander \cite{Alexander:1915} showed the case when $D$ is the unit disk and $f'(\D)$ is contained in a half-plane bounded by a straight line through the origin, and Wolff \cite{Wolff:1934} showed when $D$ is the right half-plane.
On the other hand, Tims \cite{Tims:1951} and Herzog and Piranian \cite{HerzogPiranian:1951} showed that convexity of $D$ is essential in the theorem, that is, \eqref{NWcondition} implies univalence of $f$ on $D$ if and only if $D$ is convex. 

%One example is given by $f(z) = z^{1+ \frac{\pi}{2\beta}}$ with a constant $\beta$ satisfying $\pi/2 < \beta < \pi$.
%$f$ is defined on the non-convex domain $\{z : |\arg z| < \beta\}$ and $\Re f' >0$ there, but is not univalent.

In what follows we will treat the family of functions $f \in \A$ satisfying the hypothesis of the theorem in which $D$ is the unit disk $\D$ and $c=0$.
It is denoted by $\R$, i.e.,
$$
\R := \{ f \in \A : \textup{ $\Re f'(z) >0$ for all $z \in \D$}\}.
$$
Then Theorem \ref{NWthm} states that $\R \subset \S$.
Compared with the other typical subclasses of $\S$, a geometric characterization of $\R$ is not known. 
Several geometric properties of $f(\D)$ by means of Loewner chains are observed in \cite{Hotta:oldandnew}. 
%In \cite{Hotta:oldandnew},  
Roughly speaking, $f(\D)$ is the complement of the union of the rays $\{f(e^{i\theta}) + t e^{i\theta} : t \in [0,\infty)\}$, where $f(e^{i\theta})$ is understood as the impression of the prime end at $e^{i\theta} \in \round \D$.
%\footnote{To be accurate, $f(e^{i\theta})$ is understood as the impression of the prime end at $e^{i\theta} \in \round \D$. 
%See  \cite{Hotta:oldandnew} for details.}.

A more general problem is posed of finding a domain $R \subset \C$ such that for a given simply-connected domain $D \subset \C$ the condition 
$$
f'(D) \subset R
$$
implies univalence of $f$ on $\D$.
It is studied as a first-order criterion.
%The general case of that $D$ is a simply-connected domain in $\C$ is studied as a first-order univalence criterion, a condition of the form 
%$$
%f'(D) \subset R
%$$
%which implies that $f$ is one-to-one in $D$, where $R$ is a domain in $\C$.
One can express it more generally
$$
\log f'(D) \subset R^{*}
$$
which means that $f'(z) = \exp g(z)$ where $g(D) \in R^{*}$.
It is particularly concerned with the special case in which $R^{*} = \alpha I$, where $\alpha \in \C$ and $I$ is an infinite strip parallel to the real axis with width $\pi$, i.e., $I := \{z :  a-(\pi/2) < \Im z < a+(\pi/2)\}\,(a \in \mathbb{R})$.
Theorem \ref{NWthm} gives a criterion of the case when $\alpha = 1$ and $a=c$.
For further information about first-order univalence criteria, see e.g. \cite{Gevirtz:1987, Gevirtz:1994} and more recent work \cite{ABGevirtz:2008}.

	%	+++++++++++++++++++++++++++++++++++++++++++
	%
	%		1-3
			\subsection{The aim of the paper}
	%
	%	+++++++++++++++++++++++++++++++++++++++++++

Until now, a great number of studies were dedicated to deriving sufficient conditions for the univalence of $J_{\alpha}[f]$ and $I_{\alpha}[f]$ on $\D$.
On the other hand, little seems to be known about the conditions that $J_{\alpha}[f]$ or $I_{\alpha}[f]$ produces a univalent function in $\D$ which extends to a quasiconformal map on $\C$, except that a straightforward application of the $\lambda$-lemma ($i(\lambda, z) := J_{\lambda/4}[f](z)$ forms a holomorphic motion on $(\lambda, z) \in \D \times \D$).

In this paper we discuss quasiconformal extendibility of the integral transforms $J_{\alpha}[f]$ and $I_{\alpha}[f]$ for holomorphic functions which satisfy the Noshiro-Warschawski criterion.
Various approaches using pre-Schwarzian and Schwarzian derivatives, differential subordinations and Loewner theory are taken to this problem.
%In the last section we also discuss explicit quasiconformal extensions with the criterion due to Betker  

In the last section our research contributes to constructing explicit quasiconformal extensions which are constructed by ``inverse'' counterparts of Loewner chains introduced by Betker \cite{Betker:1992}.

%In this paper univalence and quasiconformal extensibility of $J_{\alpha}[f]$ and $I_{\alpha}[f]$ on $\R$ are investigated using various approaches, pre-Schwarzian derivatives, subordination properties, holomorphic motions and Loewner theory.
%In Section 2, we collect some preliminary results on Schwarzian and pre-Schwarzian derivatives and differential subordinations of analytic functions.
%Those properties will be fully used in the later section.
%In Section 3, a sharp estimate of the norm of the pre-Schwarzian derivatives for $J_{\alpha}[f]$.
%We make use of it to 
%Section 5 and 6 will be devoted to quasiconformal extensions of the operator $J_{\alpha}[f]$.
%We will try to approach this problem from two sides. 
%One is applying the theory of holomorphic motions and the celebrated $\lambda$-lemma.
%It enables us to derive quasiconformality of $J_{\alpha}[f]$ on the full class of $\S$.
%The other is by Loewner theory and a result by Betker.
%We also consider explicit quasiconformal extension for Betker's theorem.

	%	++++++++++++++++++++++++++++++++++++++++++++++++++++++
	%
	%		Section 2
			\section{Preliminaries}
	%
	%	++++++++++++++++++++++++++++++++++++++++++++++++++++++

	%	+++++++++++++++++++++++++++++++++++++++++++
	%
	%		2-1
			\subsection{Schwarzian and pre-Schwarzian derivatives}
	%
	%	+++++++++++++++++++++++++++++++++++++++++++

As important quantities to investigate properties of functions $f$ in $\LU$, we introduce $T_f$ and $S_f$ defined by
\begin{equation*}
T_f := \frac{f''}{f'},
\hspace{20pt}
S_f := \left(\frac{f''}{f'}\right)' -\frac12\left(\frac{f''}{f'}\right)^2.
\end{equation*}
$T_f$ and $S_f$ are called the \textit{pre-Schwarzian derivative} and the \textit{Schwarzian derivative} respectively. 
These are considered as elements of the Banach space of functions $f \in \LU$, for which the norm
\begin{equation*}
\begin{array}{lll}
\dstyle ||T_f|| := \sup_{z \in \D} (1-|z|^2) |T_f|, \\[5pt]
\dstyle ||S_f|| := \sup_{z \in \D} (1-|z|^2)^{2} |S_f|,
\end{array}
\end{equation*}
is finite. 
Further, in connection with the theory of univalent functions, the following estimates are known.
Here, a homeomorphism $f$ on a domain $G$ is said to be \textit{k-quasiconformal} ($0 \leq k < 1$) if $\round_{\bar{z}}f$ and $\round_{z}f$, the partial derivatives of $f$ in $z$ and $\bar{z}$ in the distributional sense, are locally integrable on $G$ and satisfies $|\round_{\bar{z}}f| \leq k |\round_{z}f|$ almost everywhere in $G$. 
If for a given $f \in \S$ there exists a $k$-quasiconformal $F$ of $\C$ such that its restriction on $\D$ is equivalent to $f$, then $f$ is said to have a \textit{$k$-quasiconformal extension} to $\C$.
\begin{knownthm}
\label{normT}
\def\labelenumi{({\roman{enumi}}).}
Let $f \in \LU$.
\begin{enumerate}
\item If $||T_f|| \leq 1$, then $f$ is univalent in $\D$,
\item if $||T_f|| \leq k < 1$, then $f$ has a quasiconformal extension to $\C$,
\item if $f \in \S$, then $||T_f|| \leq 6$,
\item if $||S_f|| \leq 2$, then $f$ is univalent in $\D$,
\item if $f \in \S$, then $||S_f|| \leq 6$.
\end{enumerate}
\end{knownthm}
Becker showed (i) and (ii) in \cite{Becker:1972, Becker:1973}. 
The sharpness of the constant 1 in (i) is due to Becker and Pommerenke \cite{BeckerPom:1984}.
(iii) is an easy consequence of the well-known inequality $|(1-|z|^2)f''(z)/f'(z) -2\bar{z}| \leq 4$ for $f \in \S$. 
(iv) was first shown by Kraus \cite{Kraus:1932} and subsequently rediscovered by Nehari \cite{Nehari:1949}. 
Hille \cite{Hille:1949} showed that the constant 2 is the best possible one with the function $f(z)=((1+z)/(1-z))^{i \vareps}$ ($\vareps >0$), for it is not univalent for all $\vareps >0$ but $||S_f|| = 2(1+\vareps^2)$ can approach 2.
Nehari \cite{Nehari:1949} also verified the assertion (v) and the sharpness of which follows from $||S_K|| =6$ for the Koebe function $K$.

	%	+++++++++++++++++++++++++++++++++++++++++++
	%
	%		2-2
			\subsection{Subordination properties}
	%
	%	+++++++++++++++++++++++++++++++++++++++++++

For analytic functions $f$ and $g$,   it is said that \textit{$f$ is weakly subordinate to $g$} if there exists an analytic function $\omega$ which maps $\D$ into $\D$ such that $f(z) = (g \circ \omega)(z)$.
Further, if $w$ can be taken so as to fulfill $\omega(0)=0$, then $f$ is said to be \textit{subordinate} to $g$ whose relation is denoted by $f(z) \prec g(z)$.  
Below we will state two subordination properties which will play central roles in Section 3.
The first is a result on differential subordinations due to Hallenbeck and Ruscheweyh.

\begin{knownthm}[Hallenbeck and Ruscheweyh \cite{HallenRusch:1975}]\label{HR}%%%%%%
Let $p(z)$ be analytic in $\D$ with $p(0)=1$.
Let $q(z)$ be convex univalent in $\D$ with $q(0) = 1$ and suppose $p(z) \prec q(z)$.
Then for all $\gamma \neq 0$ with $\Re \gamma > 0$, we have
\begin{equation*}\label{HallenRusch}%%%%%%
\gamma z^{-\gamma} \int_{0}^{z} u^{\gamma-1}p(u) du 
\hspace{5pt}\prec
\hspace{5pt}
\gamma z^{-\gamma} \int_{0}^{z} u^{\gamma-1}q(u) du.
\end{equation*}
\end{knownthm}

\no
For example, if $f$ satisfies $\Re f'(z)(z/f(z))^{1-\gamma} >0$, then $f'(z)(z/f(z))^{1-\gamma} \prec (1+z)/(1-z)$ and Theorem \ref{HR} shows
\begin{equation*}
\left(\frac{f(z)}{z}\right)^{\gamma} \prec 1 + \frac{2\gamma}{z^{\gamma}}\int_0^z \frac{u^{\gamma}}{1-u} du.
\end{equation*}
In particular, putting $\gamma =1$ we have
\begin{equation}\label{subordinate}
\frac{f(z)}{z} \prec \frac{-z-2\log(1-z)}{z}
\end{equation}
for all $f \in \R$.
This gives the best dominant  for $\R$ because if $\phi(z) := -z-2\log(1-z)$ then $\phi'(z) = (1+z)/(1-z)$ and therefore $\phi \in \R$.

The second is a fundamental subordination principle in Geometric Function Theory. 
The original idea is due to Littlewood.
\begin{knownthm}[Kim and Sugawa {\cite[p.195]{KimSugawa:2002}}]\label{KS}%%%%%%
Let $g$ be locally univalent in $\D$. 
For an analytic function $f$ in $\D$, if $f'$ is weakly subordinate to $g'$, then we have $||T_{f}|| \leq ||T_{g}||$.
In particular, $f$ is uniformly locally univalent on $\D$.
\end{knownthm}

\no
Theorem \ref{KS} has a wide range of applications so that we might hope that one can also obtain the inequality $||S_{f}|| \leq ||S_{g}||$ for functions $f$ and $g$ such that $f$ is weakly subordinate to $g$.
However, it is show that the inequality does not always hold under this assumption.
Here we note that the Schwarz-Pick lemma shows that all analytic self-mappings $\omega$ of the unit disk satisfy
\begin{equation}\label{schwarzpick}
\frac{|\omega'(z)|}{1-|\omega(z)|^2} \leq \frac{1}{1-|z|^2}
\end{equation}
for all $z \in \D$.
\begin{prop}\label{schwarz}%%%%%%
Let $g$ be locally univalent in $\D$. 
For an analytic function $f$ in $\D$, if $f'$ is weakly subordinate to $g'$, then we have
\begin{equation}\label{schwarznorm}
||S_{f}|| \leq ||S_{g}|| + ||T_{\omega}|| \cdot ||T_{g}||,
\end{equation}
where $\omega$ is an analytic function which appears in the definition of subordination.  
In particular $f$ is uniformly locally univalent on $\D$.
\end{prop}

\begin{proof}
By assumption we have $T_{f} = T_{g} \circ \omega \cdot \omega'$ and hence \eqref{schwarzpick} implies that
\begin{eqnarray*}
\lefteqn{(1-|z|^{2})^{2} \left|\left(\frac{f''}{f'}\right)' -\frac12\left(\frac{f''}{f'}\right)^{2}\right|}\\
&\hspace{30pt}=&(1-|z|^{2})^{2} \left|\left(\frac{g''}{g'}\right)' \omega'^{2} +\frac{g''}{g'}\cdot \omega'' 
-\frac12 \left(\frac{g''}{g'} \cdot \omega'\right)^{2}\right|\\
&\hspace{30pt}\leq&
\frac{(1-|\omega|^{2})^{2}}{|\omega'|^{2}} \left|\left(\frac{g''}{g'}\right)' \omega'^{2} -\frac12 \left(\frac{g''}{g'} \cdot \omega'\right)^{2}\right| + (1-|z|^{2})\frac{1-|\omega|^{2}}{|\omega'|}\left|\frac{g''}{g'}\cdot \omega'' \right|\\
&\hspace{30pt}\leq&
||S_{g}|| + ||T_{\omega}|| \cdot ||T_{g}||.
\end{eqnarray*}
\end{proof}
The term $||T_{\omega}|| \cdot ||T_{g}||$ in \eqref{schwarznorm} is eliminated in only a few cases.
$||T_g|| = 0$ if and only if $g$ is an affine transform and then $||S_g||$ also vanishes.
Therefore $||S_f|| =0$, which implies that $f$ is a M\"obius transformation.
$||T_{\omega}|| = 0$ if and only if $\omega$ is an affine transform which is equivalent to the case that one can write $f(z) = a g(z) + b$, where $a, b \in \C$ are complex constants.

	%	++++++++++++++++++++++++++++++++++++++++++++++++++++++
	%
	%		Section 3
			\section{Pre-Schwarzian derivatives and differential subordinations for $J_{\alpha}[f]$}
	%
	%	++++++++++++++++++++++++++++++++++++++++++++++++++++++

	%	+++++++++++++++++++++++++++++++++++++++++++
	%
	%		3-1
			\subsection{Evaluation of $||T_{J_{\alpha}[f]}||$ on $\R$}
	%
	%	+++++++++++++++++++++++++++++++++++++++++++

Firstly we give a sharp estimation of the norm of $T_{J_{\alpha}[f]}$ for a function $f \in \R$ and make use of Theorem \ref{normT} to obtain the range of $|\alpha|$ which ensures univalence and quasiconformal extensibility of $J_{\alpha}[f]$.

\begin{thm}
Let $f \in \R$. Then we have the sharp estimate
\begin{equation*}
||T_{J_{\alpha}[f]}|| \leq |\alpha| \cdot h(r_{0})
\end{equation*}
where $h(r_0) \approx 1.055681$. 
Here $h$ is the function defined by
\begin{equation}\label{h(z)}
h(r) := \frac{-(1+r)^{2}}{r + 2 \log(1-r)} -\frac{1-r^{2}}{r}
\end{equation}
and $r_{0} \approx 0.329423$ is the unique root of the equation
\begin{equation}\label{2}
2(r^{2}+1)(r-1)[\log(1-r)]^{2} - 2r(r-1)^{2}\log(1-r) + r^{3}(r +3) =0.
\end{equation} 
in $r \in (0,1)$.
\end{thm}

\begin{proof}
Taking a logarithmic differentiation we have 
\begin{equation*}
||T_{J_{\alpha}[f]}|| = |\alpha| \,||T_{J[f]}||.
\end{equation*}
Then it suffices to estimate $||T_{J[f]}||$.
Let us suppose that $f \in \R$.
By \eqref{subordinate} and Theorem \ref{KS} we have
\begin{equation*}
||T_{J[f]}|| \leq \sup_{z \in \D} (1-|z|^2)\left|\frac{\phi'(z)}{\phi(z)} - \frac1{z}\right|
\end{equation*}
for all $f \in \R$, where $\phi(z) = -z-2\log(1-z)$ as defined in Section 2.2.
Then a computation shows that
\begin{eqnarray*}
\sup_{z \in \D} (1-|z|^2)\left|\frac{\phi'(z)}{\phi(z)} - \frac1{z}\right|
&=&
\sup_{z \in \D} (1-|z|^2)
\left|
\frac
{2(z + (1-z)\log(1-z))}
{(1-z)z(z+2\log(1-z))}
\right|\\
&=&
\sup_{z \in \D} \frac{1-|z|^2}{|z|}
\left|
1 + \frac{1+z}{1-z} \cdot\frac{z}{z + 2\log(1-z)}
\right|.
\end{eqnarray*}

\no
Let $\dstyle g(z) := 1 + \frac{1+z}{1-z} \cdot\frac{z}{z +  2\log(1-z)}$.
It is obvious that $g$ is symmetric with respect to the real axis. 
Next, we will show that all the coefficients of $g$ are negative.
$g$ is written as
\begin{eqnarray*}
g(z)
&=&
1 + \frac{1+z}{1-z} \cdot\frac{z}{z +  2\log(1-z)}\\
&=&
1 - \frac{1+z}{1-z}\cdot\frac{1}{1+2\sum_{n=1}^{\infty}\frac{z^n}{n+1}}\\
&=&
1 - \frac{1+z}{1-z+2\sum_{n=1}^{\infty}\frac{z^{n+1}}{n+2} - 2\sum_{n=1}^{\infty}\frac{z^{n+1}}{n+1} }\\
&=&
1 - \frac{1+z}{1-2\sum_{n=1}^{\infty}\frac{z^{n+1}}{(n+1)(n+2)} }.
\end{eqnarray*}
Thus $g$ has negative coefficients. This fact implies that
$
\sup_{z \in \D}|g(z)| = -\sup_{r \in (0,1)}g(r).
$
Therefore,
\begin{eqnarray*}
||T_{J[f]}|| 
&\leq&
\sup_{z \in \D} \frac{1-|z|^2}{|z|}
\left|
1 + \frac{1+z}{1-z} \cdot\frac{z}{z + 2\log(1-z)}
\right|\\[5pt]
&=&
\sup_{r \in (0,1)} h(r). 
\end{eqnarray*}
Simple calculation shows that $h'(r)$ has only one critical point $r_0$ in $r \in (0,1)$ which is the root of the equation \eqref{2}.
By numerical experiments, we have $r_0 \approx 0.329423$ and $h(r_0) \approx 1.055681$.
\end{proof}
%In consequence, the following is obtained.

Applying Theorem \ref{normT} to the above estimate, we can deduce the range of $|\alpha|$ of which $J_{\alpha}[f]$ is univalent in $\D$ and has a quasiconformal extension to $\C$.

\begin{cor}
Let $f \in \R$ and $k \in [0,1)$. Then,
\begin{enumerate}
\def\labelenumi{\textit{\arabic{enumi}}.}
\item If $|\alpha| \leq 1/h(r_0) \approx 0.947255$, then $J_{\alpha}[f] \in \S$,
\item If $|\alpha| < k/h(r_0)$, then $J_{\alpha}[f]$ can be extended to a $k$-quasiconformal mapping of $\C$.
\end{enumerate}
\end{cor}

	%	+++++++++++++++++++++++++++++++++++++++++++
	%
	%		3-2
			\subsection{Univalence of $J_{\alpha}[f]$ when {$\alpha \in \mathbb{R}$}}
	%
	%	+++++++++++++++++++++++++++++++++++++++++++

In the previous subsection we dealt with $J_{\alpha}[f]$ in the case that $\alpha$ is a complex number.
On the other hand, some geometric property of $J_{\alpha}[f]$ on typical subclasses of $\S$ under the restriction of $\alpha \in \mathbb{R}$ have been also investigated.
The following is a list of some fundamental results.
Here, we denote by $\K, \S^*, \CTC$ the well-known classes of convex, starlike and close-to-convex functions in $\A$, respectively.

\begin{knownthm}[Merkes and Wright \cite{MerkesWright:1971}]
Let $\alpha \in \mathbb{R}$. Then the following are true:
\begin{enumerate}
\item Let $f \in \K$. If $\alpha \in [-1, 3]$ then $J_{\alpha}[f] \in \CTC$; otherwise there exists a function $g \in \K$ such that $J_{\alpha}[g] \notin \S$.
\item Let $f \in \S^{*}$. If $\alpha \in [-\frac12, \frac32]$ then $J_{\alpha}[f] \in \CTC$; otherwise there exists a function $g \in \S^{*}$ such that $J_{\alpha}[g] \notin \S$.
\item Let $f \in \CTC$. If $\alpha \in [-\frac12, 1]$ then $J_{\alpha}[f] \in \CTC$; otherwise there exists a function $g \in \CTC$ such that $J_{\alpha}[g] \notin \CTC$.
\item Let $f \in \K$. If $\alpha \in [-\frac12, \frac32]$ then $I_{\alpha}[f] \in \CTC$; otherwise there exists a function $g \in \K$ such that $J_{\alpha}[g] \notin \S$.
\item Let $f \in \CTC$. If $\alpha \in [-\frac13, 1]$ then $I_{\alpha}[f] \in \CTC$; otherwise there exists a function $g \in \CTC$ such that $J_{\alpha}[g] \notin \S$.
\end{enumerate}
\end{knownthm}

We will show the following;
\begin{thm}\label{44}
Let $\alpha \in \mathbb{R}$ and $f \in \R$. 
If $\alpha \in [-\alpha_{0}, \alpha_{0}]$ then $J_{\alpha}[f] \in \R$; otherwise there exists a function $g \in \R$ such that $J_{\alpha}[g] \notin \R$. Here $\alpha_0 \approx 1.723078$ is defined by $\alpha_0 := \pi/2q(e^{i\theta_0})$, where
$$
q(z) := \frac{-z-2\log(1-z)}{z}
$$ 
and $\theta_0$ is the unique root of the equation $\varsigma'(\theta) - \varsigma(\theta)^2 -1 =0$ in $\theta \in (0,\pi/2)$, where $\varsigma$ is defined by
\begin{equation}
\label{varsigma}
\varsigma(\theta) := 
\frac
{\dstyle \sin \theta  + \theta - \pi}
{\dstyle \cos \theta + 2 \log \left(2\sin\frac{\theta}{2}\right)}.
\end{equation}
\end{thm}

\begin{proof}
Suppose that $f \in \R$. 
Again, we will make use of the relation \eqref{subordinate}, namely,
\begin{equation*}
J[f]' (z) \prec q(z) \prec \frac{1+z}{1-z}
\end{equation*}
for all $f \in \R$.
%Here $q$ is a convex function because $\Re [1 + zq''(z)/q'(z)] > 1 + (-1)q''(-1)/q'(-1) = -1+1/(2(2\log 2 -1))>0$ for all $z \in \D$.
Here $q$ is a convex function (see \cite[Theorem 2]{Libera:1965}).
Since $f \prec g$ implies that $f^{\alpha} \prec g ^{\alpha}$ for any $\alpha \in \mathbb{R}$ and $(J[f]')^{\alpha} = J_{\alpha}[f]'$, our problem reduces to finding the largest $\alpha_{0} \in \mathbb{R}$ such that
$
\Re [q(z)^{\alpha_{0}}] > 0
$
for all $z \in \D$.
It is equivalent to find the smallest $\beta_{0} \in \mathbb{R}$ such that the sector domain $\Delta_{\beta_0} := \{w : |\arg w| < \pi\beta_{0}/2\}$ contains $q(\D)$. 
Then $\alpha_{0} = 1/\beta_{0}$ (note that $z \in \Delta_{\beta_0}$ then $1/z \in \Delta_{\beta_0}$).

One obtains
\begin{eqnarray*}
\arg q(e^{i\theta}) 
&=& 
\arg \left[-e^{i\theta} -2 \log \left(2\sin\frac{\theta}{2}\right) -i(\theta - \pi) \right] -\theta\\
&=&
\arctan  \varsigma(\theta) -\theta
%\frac
%{\dstyle \sin \theta - \frac{\theta}{2} - \frac{3\pi}2}
%{\dstyle \cos \theta + 2 \log \left(2\sin\frac{\theta}{2}\right)}.
\end{eqnarray*}
by using $1-e^{i\theta} = - 2i  \sin(\theta/2) e^{i\theta/2}$.
Since $\round \arg q(e^{i\theta}) /\round \theta= \varsigma'(\theta)/(1 + \varsigma(\theta)^{2}) -1$, $\beta_{0}$ is one of the zeros of $\varsigma'(\theta) - \varsigma(\theta)^2 -1$.
With the aid of Mathematica, one calculates that the maximum of $\arg q(e^{i\theta})$ is attained at $\theta_{0} \approx 1.141377$.
Then $\beta_{0}= 2q(\theta_{0})/\pi \approx 0.580356$ and we conclude that $\alpha_{0}  = 1/\beta_0 \approx 1.723078$.
\end{proof}

\begin{rem}
$\R$ is preserved by the Alexander transformation $J[f]$.
\end{rem}

Theorem \ref{44} will be refined to a quasiconformal extension criterion by using the Loewner chains in Section \ref{loewner}.

	%	++++++++++++++++++++++++++++++++++++++++++++++++++++++
	%
	%		Section 4
			\section{Results for $I_{\alpha}[f]$ on $\R$}
	%
	%	++++++++++++++++++++++++++++++++++++++++++++++++++++++

We will derive some further properties of $I_{\alpha}[f]$ on $\R$.
In particular, Theorem \ref{remforI02} will be used in the later section.

\begin{thm}\label{remforI01}
Let $\alpha \in \mathbb{R}$ and $f \in \R$.
If $\alpha \in [-1, 1]$, then $I_{\alpha}[f] \in \R$.
\end{thm}

\begin{proof}
Since $I_{\alpha}[f]' = (f')^{\alpha}$, it is clear that $I_{\alpha}[f] \in \R$ when $\alpha \in [-1,1]$.
\end{proof}

\begin{thm}\label{remforI02}
Let $\alpha \in \C$.
If $|\alpha|>1$, then there exists a function $g \in \R$ such that $I_{\alpha}[g] \notin \S$.
\end{thm}

\begin{proof}
A counterexample is given by the function $\phi(z) = -z-2\log(1-z)$ which belongs to $\R$. 
In fact, it follows from the calculations that
\begin{equation*}
||S_{I_{\alpha}[\phi]}|| = 2 |\alpha|(|\alpha| + 2).
\end{equation*}
Then Theorem \ref{normT}-(v) shows that $I_{\alpha}[\phi]$ is not univalent if $|\alpha| > 1$.
\end{proof}

\begin{thm}\label{remforI03}
Let $\alpha \in \C$ and $f \in \R$.
If $|\alpha|\leq 1/2$, then $I_{\alpha}[f] \in \S$.
\end{thm}

\begin{proof}
Let $f \in \R$. Then $f'(z) \prec (1+z)/(1-z)$, and hence by Theorem \ref{KS} we obtain the sharp bound
$
||T_{f}|| \leq 2
$
for a $f \in \R$ (see also \cite[Lemma 1]{MacGregor:1963I}).
Since $||T_{I_{\alpha}[f]}|| = |\alpha|\cdot ||T_{f}||$, it follows from Theorem \ref{normT}-(i) that $I_{\alpha}[f] \in \S$ if $|\alpha|< 1/2$.
\end{proof}

			\section{Quasiconformal extension of $J_{\alpha}[f]$ with Loewner chains}
			\label{loewner}
	%
	%	++++++++++++++++++++++++++++++++++++++++++++++++++++++

In this section we will make use of the theory of Loewner chains and its applications to derive quasiconformal extension conditions for $J_{\alpha}[f]$ and $I_{\alpha}[f]$ under the class $\R$.

	%	+++++++++++++++++++++++++++++++++++++++++++
	%
	%		5-1
			\subsection{Loewner chains and inverse Loewner chains}
	%
	%	+++++++++++++++++++++++++++++++++++++++++++

Before starting our argument we describe the theory of Loewner chains and results of quasiconformal extensions due to Becker and Betker with some notations and terminology we will use.

Let $f_{t}(z) = \sum_{n=1}^{\infty}a_{n}(t)z^{n}$, $a_{1}(t) \neq 0$, be a function defined on $\D \times [0,\infty)$, where $a_{1}(t)$ is a complex-valued, locally absolutely continuous function on $[0,\infty)$.
Then $f_{t}$ is called a \textit{Loewner chain} if $f_{t}$ satisfies the following conditions;

\def\labelenumi{\it \arabic{enumi}.}
\begin{enumerate}
\item $f_{t}$ is univalent in $\D$ for each $t \geq 0$,
\item $|a_1(t)|$ increases strictly monotonically as $t$ increases, and $\lim_{t \to \infty}|a_1(t)| \to \infty$,
\item $f_{s}(\D) \subset f_{t}(\D)$ for $0 \leq s < t < \infty$.
\end{enumerate}

\no
We remark that strictly monotonicity of $|a_1(t)|$ implies that $f_s(\D) \neq f_t(\D)$ for all $0 \leq s < t < \infty$.

The key properties of Loewner chains are that $f_t$ is absolutely continuous on $t \geq 0$ for each $z \in \D$ which implies $\round_t f_t \,\,(\round_t := \round / \round t) $ exists almost everywhere on $[0,\infty)$, and satisfies the partial differential equation
\begin{equation}\label{LKPDE}
\round_t f_t(z) = z \round_z f_t(z) p(z,t)\hspace{15pt}(z \in \D,\,\textup{a.e.}\, t \geq 0),
\end{equation}
where $p(z,t)$ is analytic for all $z \in \D$ for each $t \geq 0$, measurable for all $t \geq 0$ for each $z \in \D$ and satisfies $\Re p(z,t) >0$ for all $z \in \D$ and $t \geq 0$.
We call such a function $p$ a \textit{Herglotz function}.
Further, Becker \cite{Becker:1972, Becker:1976} showed that if $p$ satisfies
\begin{equation*}
\left|\frac{1-p(z,t)}{1+p(z,t)}\right| \leq k \hspace{15pt}( z \in \D,\,\textup{a.e.}\,t \geq 0)
\end{equation*}
then $f_0$ has a $k$-quasiconformal extension to $\C$.
It enables us to derive various kinds of sufficient conditions under which a function $f \in \S$ has a quasiconformal extension (see e.g. \cite{Hotta:2009, Hotta:2010a}).

Betker introduced the following notion of inverse counterparts of Loewner chains.
Let $\omega_{t}(z) = \sum_{n=1}^{\infty}b_{n}(t)z^{n}$, $b_{1}(t) \neq 0$, be a function defined on $\D \times [0,\infty)$, where $b_{1}(t)$ is a complex-valued, locally absolutely continuous function on $[0,\infty)$.
Then $\omega_t$ is said to be an \textit{inverse Loewner chains} if 
\def\labelenumi{\it \arabic{enumi}.}
\begin{enumerate}
\item $\omega_t$ is univalent in $\D$ for each $t \geq 0$,
\item $|b_1(t)|$ decreases strictly monotonically as $t$ increases, and $\lim_{t \to \infty}|b_1(t)| \to 0$,
\item $\omega_{s}(\D) \supset \omega_{t}(\D)$ for $0 \leq s < t < \infty$,
\item $\omega_0(z) = z$ and $\omega_s(0) = \omega_t(0)$ for $0 \leq s \leq t < \infty$.
\end{enumerate}
$\omega$ also satisfies the partial differential equation:
\begin{equation}\label{bb}
\round_{t}\omega_t(z) = - z \round_{z} \omega_t(z) q(z,t)\hspace{15pt}(z \in \D,\,\textup{a.e.}\, t \geq 0),
\end{equation}
where $q$ is a Herglotz function.
Conversely, we can construct an inverse Loewner chain by means of \eqref{bb} according to the following lemma:
\begin{knownlem}[Betker \cite{Betker:1992}]
Let $q(z,t)$ be a Herglotz function.
Suppose that $q(0,t)$ be locally integrable in $[0, \infty)$ with $\int_0^{\infty} \Re q(0, t) dt = \infty$. 
Then there exists an inverse Loewner chain $w_t$ satisfying \eqref{bb}.
\end{knownlem}
Applying the notion of an inverse Loewner chain, we obtain a generalization of Becker's result.

\begin{knownthm}[Betker \cite{Betker:1992}]
Let $k \in (0,1]$.
Let $f_t$ be a Loewner chain satisfying \eqref{LKPDE} with
\begin{equation*}
\left|
\frac{p(z,t) -\closure{q(z,t)}}{p(z,t) + q(z,t)}
\right|
\leq k <1\hspace{20pt}
(z \in \D, \textup{a.e.}\, t \geq 0)
\end{equation*}
where $q(z,t)$ is a Herglotz function.
Let $\omega_t$ be the inverse Loewner chain which is generated by $q$ with \eqref{bb}. 
Then $f_t$ and $\omega_t$ are continuous and injective on $\closure{\D}$ for each $t \geq 0$, and $f_0$ has a $k$-quasiconformal extension $\Phi : \CC \to \CC$ which is defined by
\begin{equation}\label{betker}%%
\Phi\left(\frac{1}{\closure{\omega_t(e^{i\theta})}}\right) = f_t(e^{i\theta}) \hspace{20pt} (\theta \in [0,2\pi), t \geq 0).
\end{equation}
\end{knownthm}
The case $q(z,t) = 1$ reflects Becker's theorem. In this case $\omega_t(z) = e^{-t}z$.
Further, if $\omega$ is obtained from the choice $q = p$, then we have the following corollary:

\begin{knowncor}[Betker \cite{Betker:1992}]
\label{Betkercor}
Let $\gamma \in (0,1]$.
Suppose that $f_t$ is a Loewner chain for which $p$ in \eqref{LKPDE} satisfies the condition 
\begin{equation*}
|\arg p(z,t)| \leq \frac{\gamma \pi}{2}\hspace{20pt}(z \in \D,\,\textup{a.e.}\,t \geq 0).
\end{equation*}
Then $f_{t}$ admits a continuous extension to $\closure{\D}$ for each $t \geq 0$ and the map defined by \eqref{betker} is a $\sin (\gamma \pi/2)$-quasiconformal extension of $f_{0}$ to $\C$.
\end{knowncor}

In contrast to Becker's quasiconformal extension theorem mentioned above, the theorem due to Betker does not always give a quasiconformal extension explicitly.
The reason is due to the fact that in general it is difficult to express an inverse Loewner chain $\omega_t$ which has the same Herglotz function as a given Loewner chain $f_t$ in an explicit form.

In more detail, let $f_t$ be a given Loewner chain and $p(z,t)$ be a Herglotz function associated with $f_t$ by \eqref{LKPDE}.
Fix an arbitrary $T > 0$, and define a Herglotz function $q(z,t)$ by
\begin{equation}
\label{q(z,t)}
q(z,t) := \left\{
\begin{array}{lll}
p(z,T -t), & t \in [0, T]\\[3pt]
1, &  t \in (T, \infty). 
\end{array}
\right.
\end{equation}
%Then there exists a Loewner chain $h_t$ with the equation $\round_t h_t(z) = z \round_z h_t(z) q(z,t)$.
It is known that there exists a Loewner chain $h_t$ with the equation $\round_t h_t(z) = z \round_z h_t(z) q(z,t)$.
One can see that $g_t(z)$ defined by
\begin{equation}
\label{g_t(z)}
g_t(z) := \left\{
\begin{array}{lll}
(h_{T}^{-1} \circ h_t)(z), & t \in [0, T]\\[3pt]
e^{T -t}, &  t \in (T, \infty),
\end{array}
\right.
\end{equation}
is also a Loewner chain whose Herglotz function is $q$.
Such $g_t$ is uniquely determined by the condition $g_{T}(z) = z$.
Therefore $g_t$ is the unique solution of the differential equation
\begin{equation*}
\round_t g_t(z) = z \round_z g_t(z) p(z, T -t)
\end{equation*}
for all $z \in \D$ and $t \in [0,T]$.
Hence $\omega_t := g_{T -t}$ is defined on $z \in \D$, $t \in [0,T]$ and satisfies $\round_{t}\omega_t(z) = - z \round_{z} \omega_t(z) p(z,t)$.
It is also easily seen that $\omega_0(z) =z,\,\omega_t(0) =0,\,\omega_s(\D) \supset \omega_t(\D)$ and $b_1(t)$ is monotonically decreasing with $|b_1| \to 0$ as $t \to \infty$.
Since $T$ is arbitrary, we obtain our desired inverse Loewner chain.

%Since $(d/dt)\log|b_1(t)| = $

%Further, we can see that $\omega_0(z) =z$,\, 

%is defined at least for all $t \in [0,T]$

%Our desired inverse Loewner chain $\omega_t$ follows from $\omega_t := g_{T -t}$ and repeating this procedure. 
%In consequence, in order to obtain the concrete expression of $\omega_t$ we need to describe $h_t$ and $h_t^{-1}$ by a given $f_t$, and it is not always possible.

The above argument indicates that in order to obtain the concrete expression of $\omega_t$ we need to write $h_t$ and $h_t^{-1}$ by a given $f_t$, and it is not always possible.
Loewner chains for spirallike functions are one of the few known cases in which this method works well.
Here $f \in \A$ is said to be \textit{$\lambda$-spirallike} ($\lambda \in (-\pi/2, \pi/2)$) if $f$ satisfies
\begin{equation*}
\Re \left\{ e^{-i\lambda}\frac{zf'(z)}{f(z)} \right\}>0
\end{equation*}
for all $z \in \D$.
We know that $f_t(z) = e^{e^{i\lambda} t} f(z)$ describes an expanding flow for $\lambda$-spirallike domains.
In this case the corresponding inverse Loewner chain $\omega_t$ can be written explicitly by 
\begin{equation*}
\omega_t(z) := f^{-1}(e^{-e^{i\lambda} t} f(z)).
\end{equation*}
Let $\alpha \in (-\pi/2, \pi/2)$ be given. 
Suppose $|\arg (zf'(z)/f(z)) -\lambda| < \pi\alpha/2$.
Then by Corollary \ref{Betkercor} $f$ has a continuous extension to $\closure{\D}$, and the function $\Phi : \C \to \C$,
\begin{equation}
\label{explicit01}
\left\{
\begin{array}{ll}
\Phi(z) = f(z), & z \in \closure{\D}\\[5pt]
\dstyle\Phi\left(\frac{1}{\closure{f^{-1}(e^{-e^{i\lambda} t} f(e^{i\theta}))}}\right) = e^{e^{i\lambda} t} f(e^{i\theta}), & \theta \in [0,2\pi), t \geq 0.
\end{array}
\right.
\end{equation}
defines a $\sin(\pi \alpha/2)$-quasiconformal extension of $f$.
If $z =1/\closure{f^{-1}(e^{-e^{i\lambda} t} f(e^{i\theta}))}$ we have 
\begin{equation}
\label{w-equation}
f\left(\frac{1}{\closure{z}}\right) = e^{-e^{i\lambda} t} f(e^{i\theta})
\end{equation}
and hence \eqref{explicit01} is expressed by
\begin{equation}
\label{explicit02}
\Phi(z) =
\left\{
\begin{array}{ll}
f(z), & z \in \closure{\D}\\[5pt]
\dstyle \frac{(f(e^{i\theta}))^2}{f(1/\bar{z})}, & z \in \C \backslash\closure{\D},
\end{array}
\right.
\end{equation}
where $f(e^{i\theta})$ is uniquely determined by the equation $\arg_{\lambda} f(1/\closure{z}) = \arg_{\lambda} f(e^{i\theta})$ which is deduced by \eqref{w-equation}, where $\arg_{\lambda}$ represents the $\lambda$-argument (for details, see \cite{KimSugawa:pre01}).
The function \eqref{explicit02} is the same as given in \cite{Sugawa:2012a}.

	%	+++++++++++++++++++++++++++++++++++++++++++
	%
	%		5-2
			\subsection{Results}
	%
	%	+++++++++++++++++++++++++++++++++++++++++++

Several conditions under which $f \in \R$ has a quasiconformal extension to the complex plane are known.
One of the remarkable results is due to Chuaqui and Gevirtz \cite{ChuaquiGevirtz:2003} who gave the necessary and sufficient condition under which $f(\D)$ can be a quasidisk by introducing the notion of \textit{property M}.
Comparing to it, our results provide quantitative estimates for the dilatations of quasiconformal extensions.

A Loewner chain for the class $\R$ is simply given by
\begin{equation*}
\label{loewnerforR01}
f_t(z) := f(z) +  t z.
\end{equation*}
In fact, a straightforward calculation shows that 
\begin{equation*}
\label{1/p}
\frac{1}{p(z,t)} = \frac{\round_t f_t(z)}{z\round_z f_t(z)} = f'(z) + t.
\end{equation*}
If we assume that $|\arg f'(z)| \leq \gamma\pi/2$ for a fixed constant $\gamma \in (0,\pi/2]$, then it follows from Corollary \ref{Betkercor} that $f$ has a $\sin (\gamma \pi/2)$-quasiconformal extension to $\C$.
Consequently we will obtain the following.

\begin{thm}
\label{qcextforR}
Let $f \in \A$ and $\gamma \in [0,1)$. 
If $|\arg f'(z)| \leq \gamma \pi /2$ for all $z \in \D$, then $f$ belongs to $\R$ and has a $\sin (\gamma \pi/2)$-quasiconformal extension to $\C$.
\end{thm}
As we have seen above, in this case it does not seem to obtain an explicit quasiconformal extension by \eqref{betker} because there is no feasible means to find a Loewner chain $h_t$ whose Herglotz function is given by \eqref{q(z,t)} with $q(z,t) = f'(z) + t$ and its inverse function $h_t^{-1}(z)$ to define $g_t$ by \eqref{g_t(z)}.

%We finish this section to show the following results. 
\begin{thm}
Let $f \in \mathcal{R}$.
Let $\beta_0 \approx 0.580356$ and $\alpha_0 = 1/\beta_0 \approx 1.723078$ be constants which are given in Subsection 3.2 and $\alpha \in (-\alpha_0, \alpha_0)$ be fixed.
Then $J_{\alpha}[f]$ has a $\sin(|\alpha| \beta_0\pi/2)$-quasiconformal extension to $\C$. 
\end{thm}
\begin{proof}
We have shown in Subsection 3.2 that if $f \in \R$ then $\{f(z)/z : z \in \D\}$ lies in the sector domain $\Delta_{\beta_0} =\{w : |\arg w| < \pi\beta_0 /2\}$.
It implies that $(f(z)/z)^{\alpha} =  J_{\alpha}[f]'(z) \in \Delta_{\alpha \beta_0}$ for all $z \in \D$, and therefore
\begin{equation*}
\left|\arg J_{\alpha}[f]'(z)\right| \leq \frac{|\alpha|\beta_0\pi}{2}.
\end{equation*}
Hence Theorem \ref{qcextforR} yields our assertion.
\end{proof}

\begin{thm}
Let $f \in \R$.
Then for a fixed $\alpha \in (-1,1)$, $I_{\alpha}[f]$ has a $\sin(|\alpha|\pi/2)$-quasiconformal extension to $\C$.
On the other hand, if $\alpha$ lies on $(-\infty, -1] \cup [1, \infty)$, then there exists a function $g \in \R$ such that $I_{\alpha}[g]$ does not have any quasiconformal extension.
\end{thm}

\begin{proof}
Let $f \in \R$.
Since $|\arg I_{\alpha}[f]'| < |\alpha|\pi/2$, applying Theorem \ref{qcextforR} we conclude that $f$ has a $\sin(|\alpha|\pi/2)$-quasiconformal extension to $\C$.

As for the second statement of the theorem, by Theorem \ref{remforI02} it suffices to consider the case when $\alpha = 1$ or $\alpha = -1$. 
If $\alpha =1$, then our statement easily follows because $I_{1}[\phi] = \phi(z) = -z-2\log(1-z) \in \R$ map $\D$ onto a domain which is not a quasicircle.
One can similarly show it in the case when $\alpha=-1$ with the counterexample $\psi(z) = -z+2\log(1+z)$.
\end{proof}

\section*{Acknowledgement}
The authors would like to thank the referee for reading a manuscript of the paper carefully and pointing out some errors which improved this paper.

%ReferencesReferencesReferencesReferencesReferencesReferencesReferences
%ReferencesReferencesReferencesReferencesReferencesReferencesReferences
%ReferencesReferencesReferencesReferencesReferencesReferencesReferences
%ReferencesReferencesReferencesReferencesReferencesReferencesReferences

\bibliographystyle{amsalpha}
\def\cprime{$'$}
\providecommand{\bysame}{\leavevmode\hbox to3em{\hrulefill}\thinspace}
\providecommand{\MR}{\relax\ifhmode\unskip\space\fi MR }
% \MRhref is called by the amsart/book/proc definition of \MR.
\providecommand{\MRhref}[2]{%
  \href{http://www.ams.org/mathscinet-getitem?mr=#1}{#2}
}
\providecommand{\href}[2]{#2}

%ReferencesReferencesReferencesReferencesReferencesReferencesReferences
%ReferencesReferencesReferencesReferencesReferencesReferencesReferences
%ReferencesReferencesReferencesReferencesReferencesReferencesReferences
%ReferencesReferencesReferencesReferencesReferencesReferencesReferences

\end{document}